\documentclass[11pt]{amsart}
\usepackage{amsmath}
\usepackage{amssymb}
\usepackage{amsfonts}
\usepackage{latexsym}
\usepackage{fancyhdr}
\usepackage{graphics}
\usepackage{graphicx}
\usepackage{color}
\newcommand{\mb}{\mathbb}
\newcommand{\mc}{\mathcal}

\newcommand{\M}{\mc M}

\newcommand{\D}{\mc D}

\def\H{\mc H}

\newcommand{\SC}{\mathcal{S}({\mb R})}
\newcommand{\SSS}{\mathcal{S}^\times({\mb R})}

\newcommand{\LDD}{{\mc L}(\D,\D^\times)}
\newtheorem{defn}{Definition}[section]
\newtheorem{prop}[defn]{Proposition}
\newtheorem{thm}[defn]{Theorem}

\theoremstyle{definition}
\newtheorem{example}[defn]{Example}
\newtheorem{rem}[defn]{Remark}
\def\x{\relax\ifmmode {\mbox{*}}\else*\fi}
\newcommand{\beex}{\begin{example}$\!\!${\bf }$\;$\rm }
\newcommand{\enex}{ \end{example}}
\newcommand{\berem}{\begin{rem}$\!\!${\bf }$\;$\rm }
\newcommand{\enrem}{ \end{rem}}

\newcommand{\bedefi}{\begin{defn}$\!\!${\bf }$\;$\rm }
\newcommand{\findefi}{\end{defn}}

\newcommand{\ip}[2]{\langle {#1}|{#2}\rangle}
\newcommand{\RN}{{\mb R}}

\newcommand{\LD}{{\mathcal L}^\dagger(\D)}

\begin{document}

\title[Some notes about distribution frame multipliers]{Some notes about distribution frame multipliers}%

\author{Rosario Corso}%
\address{Dipartimento di Matematica e Informatica, Universit\`a degli Studi
di Palermo, I-90123 Palermo (Italy)} \email{rosario.corso02@unipa.it}

\author{Francesco Tschinke}%
 \email{francesco.tschinke@unipa.it}

\begin{abstract}
\noindent
Inspired by a recent work about distribution frames, the definition of multiplier operator is extended in the rigged Hilbert spaces setting and a study of its main properties is carried on.  In particular, conditions for the density of domain and boundedness are given.  The case of Riesz distribution bases is examined in order to develop a symbolic calculus.

\end{abstract}
\maketitle
\section{Introduction}

Bessel multipliers were introduced by Balazs in \cite{Balazs_basic_mult} and they became objects of several works \cite{bal_invers,bal_mult,bal_uncond,bal_Riesz,bal_dual_frame2,bal_dual_frame,bal_survey}. Multipliers have been studied also in particular cases (\cite{Ben_Pfander,FeiNow,Groechenig}) and found applications in physics, signal processing, acoustics and mathematics. To define them we need to recall some notions (see \cite{christensen2,heil}).

Let $\H$ be a Hilbert space with inner product $\ip{\cdot}{\cdot}$ and norm $\|\cdot\|$. A sequence $\phi=\{\phi_n\}_{n\in \mb N}$ is a {\it Bessel sequence} of $\H$ with upper bound $B>0$ if
\begin{equation*}
\label{up_frame}
\sum_{n\in \mb N} |\ip{f}{\phi_n}|^2\leq B\|f\|^2, \qquad \forall f\in \H.
\end{equation*}
A sequence $\phi=\{\phi_n\}_{n\in \mb N}$ is a {\it (discrete) frame} if there exist $A,B>0$ such that
$$
A\|f\|^2 \leq \sum_{n\in \mb N} |\ip{f}{\phi_n}|^2\leq B\|f\|^2, \qquad \forall f\in \H.
$$

Let $\varphi=\{\varphi_n\}_{n\in \mb N},\psi=\{\psi_n\}_{n\in \mb N}$ be two sequences of $\H$ and $m:\mb N\to \mb C$. The operator $M_{m,\varphi,\psi}$ defined by
$$
M_{m,\varphi,\psi}f=\sum_{n\in \mb N} m_n\ip{f}{\psi_n}\varphi_n,
$$
and with domain the subspace of $f\in \H$ such that $\sum_{n\in \mb N} m_n\ip{f}{\psi_n}\varphi_n $ is convergent, is called the {\it multiplier} of $\varphi,\psi$ with symbol $m$. When $\varphi,\psi$ are Bessel (frame) sequences and $m$ is a bounded sequence, then $M_{m,\varphi,\psi}$ is defined on $\H$, bounded, and it is called a {\it Bessel (frame) multiplier}.

Independently in \cite{AAG_paper} and in \cite{kaiser},  the notion of continuous frame was introduced as generalization of discrete frame and later in \cite{balaszetal} the correspondent notion of Bessel continuous (frame) multiplier was formulated, which we now recall. The setting involves a measure space $(X,\mu)$ with  positive measure $\mu$. A map $F:x\in X\to F_x\in \H$ is called a {\it continuous frame} with respect to $(X,\mu)$ if
\begin{enumerate}
	\item[(i)] $F$ is weakly measurable, i.e. $f\mapsto\ip{f}{F_x}$ is $\mu$-measurable for every $f\in \H$;
	\item[(ii)] there exist $A,B>0$ such that
	\begin{equation}
	\label{cont_frame}
	A\|f\|^2\leq \int_X |\ip{f}{F_x}|^2d\mu\leq B\|f\|^2, \qquad \forall f\in \H.
	\end{equation}
\end{enumerate}
A weakly measurable map $F:x\in X\to F_x\in \H$ is called a {\it Bessel continuous map} with respect to $(X,\mu)$ if the second inequality in \eqref{cont_frame} holds. 
Let $F,G:X\to \H$ be Bessel continuous maps and $m\in L^\infty(X,\mu)$. An operator $M_{m,F,G}$ can be weakly defined by
$$
\ip{M_{m,F,G}f}{g}=\int_X m(x)\ip{f}{F_x}\ip{G_x}{g}d\mu, \qquad f,g\in \H.
$$
This operator is called the {\it Bessel continuous multiplier} of $F,G$ with symbol $m$  (and {\it continuous frame multiplier} if $F,G$ are in addition continuous frames). Among continuous frame multipliers one can find time-frequency localization operators \cite{daub1,daub2} (also called short-time Fourier transform multipliers) and Calder\'{o}n-Toeplitz operators \cite{roch,roch2}. 

Recently, a notion of frame (and related topics such as bases,  Bessel maps, Riesz bases and Riesz-Fischer maps) in space of distributions is appeared in \cite{TTT,tschinke}, involving a {\it rigged Hilbert space}, or {\it Gel'fand triplet}, i.e.\ a triple  $\D[t]\subset \H \subset \D^\times[t^\times]$, where  $\D[t]$ is a dense subspace of $\H$ endowed with a locally convex topology $t$, stronger than the one induced by the Hilbert norm and $\D^\times[t^\times]$ is the  conjugate dual of $\D[t]$ with the strong dual topology $t^\times$. If $\D[t]$ is reflexive, then the inclusions are dense and continuous. Analogous concepts in rigged Hilbert spaces and for the discrete case have been considered also in \cite{gb_ct_riesz}. The aim of this paper is then to give a correspondent notion of multipliers of distribution maps. A preliminary, but very confined, study about distribution multipliers actually was given in \cite{TTT}.

In contrast with a large part of the current literature, we will not pay attention to bounded multipliers only. To give an example of the importance of unbounded multipliers, we mention \cite{bit_jmp,Bag_sesq} where they were used as tools to define non-selfadjiont hamiltonians. Sufficient conditions for operators to be written as multipliers with a fixed sequence have been given in \cite{bellcorso,gavruta}.

The paper is organized as follows. We start by recalling some preliminaries in Section \ref{sect_2}. Then, in Section \ref{sect_3}, we give the definitions of distribution multipliers. They can actually be formulated in two different ways, i.e.\ as operators from $\D$ to $\D^\times$ or as operators on $\H$. Questions about density of domain and closedness of unbounded multipliers are discussed in Section \ref{sect_4}, while Riesz distribution multipliers are studied and some results about symbolic calculus is obtained in Section \ref{sect_5}. 

\section{Preliminary definitions and facts}
\label{sect_2}

Throughout the paper $\H$ indicates a Hilbert space with inner product $\ip{\cdot}{\cdot}$ and norm $\|\cdot\|$. We denote by $D(T)$ and $R(T)$ the domain and the range of an operator $T:D(T)\subset \H \to \H$. If $T$ is densely defined, then we write $T^*$ for its adjoint.

A sequence $\phi=\{\phi_n\}_{n\in \mb N}\subset \H$ is called {\it total} if $\ip{f}{\phi_n}=0$ for every $n\in \mb N $ implies that $f=0$. In particular, discrete frames are total sequences. A {\it Riesz basis} $\phi$  is a total sequence satisfying for some $A,B>0$
$$
A \sum_{n\in \mb N} |c_n|^2 \leq \left \| \sum_{n\in \mb N} c_n \phi_n \right \|^2 \leq B\sum_{n\in \mb N} |c_n|^2, \qquad\forall \{c_n\}\in \ell_2(\mb N),
$$
where $\ell_2(\mb N)$ is the usual space of square integrable complex sequences.

  Let $\D[t]$ be  a dense subspace of $\H$ endowed with a locally convex topology $t$, stronger than the topology induced by the Hilbert norm.
The vector space of all continuous conjugate
linear functionals on $\D[t]$ (the conjugate dual of $\D[t]$)
is denoted by $\D^\times[t^\times]$, and is endowed with the {\em strong dual topology} $t^\times$, defined by the seminorms
\begin{equation*}\label{semin_Dtimes}
q_\M(F)=\sup_{g\in \M}|\ip{F}{g}|, \quad F\in \D^\times,
\end{equation*}
where $\M$ is a bounded subsets of $\D[t]$. With a well-known identification procedure (see \cite{horvath}),  $\H$  is considered as subspace of $\D^\times[t^\times]$. The triplet
$$
\D[t] \subset  \H \subset\D^\times[t^\times],
$$
is called {\em rigged Hilbert space} or {\em Gel'fand triplet} \cite{gelf3,gelf}. If $\D[t]$ is reflexive $\H$ is continuously and densely embedded in $\D^\times[t^\times]$. Denoting by $\hookrightarrow$ the continuous and dense embedding,  the triple
is also denoted by
\begin{equation*}\label{eq_one_intr}
\D[t] \hookrightarrow  \H \hookrightarrow\D^\times[t^\times].
\end{equation*}
In this way, the sesquilinear form $B( \cdot , \cdot )$ which puts $\D$
and $\D^\times$ in duality extends the inner product of
$\H$;
 i.e. $B(\xi, \eta) = \ip{\xi}{\eta}$, for every $\xi, \eta \in \D$: we adopt the symbol $\ip{\cdot}{\cdot}$ for both of
 them. 
 
Let us denote by $\LDD$ the vector space of all continuous linear maps from $\D[t]$ into  $\D^\times[t^\times]$ (\cite{ait_book}). If $\D[t]$ is reflexive, it is possible introduce an involution $X \mapsto X^\dag$  in $\LDD$  by the identity:
$$ \ip{X^\dag \eta}{ \xi} = \overline{\ip{X\xi}{\eta}}, \quad \forall \xi, \eta \in \D.$$  Hence, in this case, $\LDD$ is a $^\dagger$-invariant vector space.

If $\D[t]$ is a {smooth} space (e.g., Fr\'echet and reflexive), then $\LDD{}$ is a quasi *-algebra over $\LD$ (Definition 2.1.9 of \cite{ait_book}).

We also denote by ${\mc L}(\D)$ the algebra of all continuous linear operators $Y:\D[t]\to \D[t]$ and by ${\mc L}(\D^\times)$ the algebra of all continuous linear operators $Z:\D^\times[t^\times]\to \D^\times[t^\times]$.
If $\D[t]$ is reflexive, for every $Y \in {\mc L}(\D)$ there exists a unique operator $Y^\times\in {\mc L}(\D^\times)$, the {\it adjoint} of $Y$, such that
$$ \ip{F}{Yg} = \ip{Y^\times F}{g}, \quad\forall F \in \D^\times, g \in \D.$$ In similar way an operator $Z\in {\mc L}(\D^\times)$ has an adjoint $Z^\times\in {\mc L}(\D)$ such that $(Z^\times)^\times=Z$. We denote by $\LD$ the algebra of all closable operators $A$ in $\H$ such that $D(A)=\D$, $D(A^*)\supseteq\D$, and $A$,  $A^*$ leave $\D$ invariant. With the involution $A\mapsto A^*\upharpoonright_{\D}=A^\dag$, $\LD$ is a *-algebra.

In this paper  $(X,\mu)$ denotes a measure space with a $\sigma$-finite positive measure $\mu$. We recall that a measurable set $A\subseteq X$ is called an {\it atom} if $\mu(A)>0$ and for every $B\subseteq A$ we have either $\mu(B)=0$ or $\mu(B)=\mu(A)$. 
A measure space $(X,\mu)$ is called {\it atomic} if there
exists a partition $\{A_n\}_{n\in \mb N}$ of $X$ consisting of atoms and sets of measure zero. We write $L^1(X,\mu),L^2(X,\mu)$ and $L^\infty(X,\mu)$ for the usual spaces of (classes of) measurable functions. Moreover, $\|m\|_\infty$ denotes the essential supremum of $m\in L^\infty(X,\mu)$.  For simplicity, we write $L^1(\mb R)$, $L^2(\mb R)$ and $L^\infty(\mb R)$ when we assume the Lebesgue measure. The Fourier transform of $f\in L^1(\mb R)$ is defined as 
	$\widehat f(\gamma)=\int_{\mb R}  f(x)e^{-2\pi i \gamma x}dx$ and it extends to a unitary operator of $L^2(\mb R)$ in a standard way.
In this paper we consider weakly measurable maps:  given a measure space  $(X,\mu)$ with $\mu$ a $\sigma$-finite positive measure,
  $\omega: x\in X\to \omega_x \in \D^\times$ is a {\it weakly measurable map} if, for every $f \in \D$, the complex valued function $x \mapsto \ip{f}{\omega_x}$ is $\mu$-measurable. 
  If not otherwise specified, throughout the paper we will work with a fixed rigged Hilbert space $\D[t] \subset  \H \subset\D^\times[t^\times]$ with $\D[t]$ reflexive and a measure space $(X,\mu)$ as described before. 
  We start by recalling simple definitions about weakly measurable maps. Since the form which puts $\D$ and $\D^\times$ in conjugate duality is an extension of the inner product of $\H$, we write $\ip{f}{\omega_x}$ for $\overline{\ip{\omega_x}{f}}$,  $f \in \D$.

\begin{defn}[{\cite[Definition 2.2]{TTT}}]
\label{tandg}
Let  $\omega: x\in X\to \omega_x \in \D^\times$ be a weakly measurable map, then:
\begin{itemize}
\item[(i)] $\omega$ is \textit{total} if, $f \in \D $ and $\ip{f}{\omega_x}=0$  $\mu$-a.e. $x \in X$ implies $f=0$;
\item[(ii)]$\omega$ is \textit{$\mu$-independent} if the unique $\mu$-measurable function $\xi:X\to \mb C$ such that: \\$\int_X \xi(x)\ip{g}{\omega_x} d\mu=0$, for every $g \in \D$, is $\xi(x)=0$ $\mu$-a.e.
\end{itemize}
\end{defn}

\begin{defn}[{\cite[Definition 3.2]{TTT}}] A weakly measurable map $\omega$ is a {\em Bessel distribution map} (briefly: Bessel map) if for every $f \in \D$,
$ \int_X |\ip{f}{\omega_x}|^2d\mu<\infty$.
\end{defn}
It is convenient to consider $\D[t]$ as a Fr\'{e}chet space because of the following proposition.

\begin{prop}[{\cite[Proposition 3.1]{TTT}}]\label{prop2} Let $\D[t]$ be a Fr\'{e}chet space and $\omega: x\in X \to \omega_x\in \D^\times$ a weakly measurable map. The following statements are equivalent.
\begin{itemize}
\item[(i)]
$\omega$ is a  Bessel  map.
\item[(ii)] There exists a continuous seminorm $p$ on $\D[t]$ such that
\begin{equation*}\label{eqn_bessel1}\left( \int_X |\ip{f}{\omega_x}|^2d\mu\right)^{1/2}\leq p(f), \quad \forall f \in \D.\end{equation*}
 \item[(iii)] For every bounded subset $\mathcal M$ of $\D$ there exists $C_{\mathcal M}>0$ such that
\begin{equation*}
\label{disbessel}
\sup_{f\in\mathcal M}{\Bigr |}\int_X\xi(x)\ip{\omega_x}{f}d\mu{\Bigl |}\leq C_{\mathcal M}\|\xi\|_2, \quad \forall \xi\in L^2(X,\mu).
\end{equation*}
  \end{itemize}
\end{prop}
As a consequence of the previous proposition, we have \cite{TTT}:
\begin{itemize}
\item
the  conjugate linear functional on $\D$:
$$ {\Lambda^\xi_\omega}:=\int_X\xi(x){\omega_x}d\mu$$
is defined in weak sense, and is continuous, i.e. $\Lambda_\omega^\xi\in \D^\times[t^\times]$;
\item
the {\em synthesis operator}  $T_\omega:L^2(X, \mu)\to \D^\times[t^\times]$  defined by
$ T_\omega: \xi \mapsto {\Lambda^\xi_\omega}$
 is continuous;
\item
the {\em analysis operator} $T_\omega^\times: \D[t]\to L^2(X, \mu)$ defined by
$(T_\omega^\times f)(x) =\ip{f}{\omega_x}$
is continuous;
\item
the {\it frame operator} $S_\omega:\D[t]\rightarrow\D^\times[t^\times]$, $S_\omega:=T_\omega T_\omega^\times$ is continuous, i.e. $S_\omega\in\LDD$.
\end{itemize}

 We will often work with a special class of Bessel maps which is defined as follows.
\begin{defn}[{\cite[Definition 3.2]{TTT}}]
\label{bbounded}
A Bessel distribution map $\omega$ is called {\it bounded Bessel map} if there exists $B>0$ such that
$$
\int_X|\ip{\omega_x}{f}|^2d\mu\leq B\|f\|^2, \quad\forall f\in\D.
$$
\end{defn}
If $\omega$ is a bounded Bessel map, with a limit procedure, we have \cite{TTT}:
\begin{itemize}
\item
$\Lambda_\omega^\xi$ is bounded in $\D(\|\cdot\|)$, then it has a bounded extension ${\tilde\Lambda}_\omega^\xi$ to $\H$;
\item
the synthesis operator $T_\omega$ has range in $\H$, it is bounded and $\|T_\omega\|\leq \sqrt{B}$;
\item
the (Hilbert adjoint) operator  $T_\omega^*:\H\rightarrow L^2(X,\mu)$ extends $T_\omega^\times$ and $T_\omega^*f=\eta$, where $\eta$ is the limit in $L^2(X,\mu)$ of the functions $T_\omega^\times f_n: x\mapsto \ip{f_n}{\omega_x}$, where $\{f_n\}$ is a sequence in $\D$ converging to $f$ (we will also denote the function $T_\omega^*f$ by $x\mapsto \ip{f}{\check\omega_x}$ for $f\in \H$, i.e. we consider $\check \omega_x$ as an `extension' of the linear functional $\omega_x$);
\item
the operator $\widehat{S}_\omega=T_\omega T_\omega^*$ is bounded and it is an extension of $S_\omega$.
\end{itemize}

 Now it is time to recall a notion of frames in the distribution context  ({Definition 3.6 of \cite{TTT}}).

\begin{defn}[{\cite[Definition 3.6]{TTT}}] \label{defn_distribframe}  Let $\D[t]\subset\H\subset \D^\times[t^\times]$ be a rigged Hilbert space, with $\D[t]$ a reflexive space and $\omega$ a weakly measurable map.
We say that $\omega$ is a {\em  distribution frame} if there exist $A,B>0$ (called {\it frame bounds}) such that
\begin{equation*} \label{eqn_frame_main1} A\|f\|^2 \leq \int_X|\ip{f}{\omega_x}|^2d\mu \leq B \|f\|^2, \quad \forall f\in \D. \end{equation*}
Moreover, we say that
\begin{itemize}
	\item[a)]
	$\omega$ is a {\it tight} distribution frame if we can choose $A = B$ as frame bounds of $\omega$;
	\item[b)]
	$\omega$ is a {\it Parseval} distribution frame if $A = B = 1$ are frame bounds of $\omega$.
\end{itemize}
\end{defn}
If $\omega$ is a distribution frame, then the frame operator $\hat{S}_\omega$  satisfies  the inequalities
$$ A\|f\| \leq \| \hat{S}_\omega f\| \leq B\|f\|,\quad \forall f\in \H. $$
Since  $\hat{S}_\omega$ is symmetric, this  implies that $\hat{S}_\omega$ has a bounded inverse $\hat{S}_\omega^{-1}$ everywhere defined in $\H$.

A Parseval distribution frame satisfies Definition \ref{defn_distribframe} with $S_\omega=I_\D$, the identity operator of $\D$, and $\widehat{S}_\omega=I_{\H}$, the identity operator of $\H$.

\begin{example}[{\cite[Example 3.18]{TTT}}] \label{ex_delta}Let us consider the rigged Hilbert space
	$$
	\SC\subset L^2(\mb R)\subset\SSS
	$$
	where $\SC$ is the Schwartz space of rapidly decreasing $C^\infty$-functions on $\mb R$ and the conjugate dual $\SSS$ is the space of tempered distributions.  
	Let $\delta$ be the weakly measurable map $\delta: x\in {\mb R}\to \delta_x\in \SSS$, where $\delta_x$ stands for the $\delta$ distribution centered at $x$. As known, $\delta_x$ acts in the following way
	$\ip{\delta_x}{\phi}=\overline{\phi(x)} $, for every $\phi\in \SC$.
	Then one trivially has
	$$\int_\RN |\ip{\phi}{\delta_x}|^2dx= \int_\RN|\phi(x)|^2 dx =\|{\phi}\|^2,\quad \forall \phi\in \SC,$$
	where $\|{\phi}\|^2$ is the norm in $L^2(\mb R)$; hence, $\delta$ is a Parseval frame.
\end{example}

We note that this example is based on the measure space $(X,\mu)=(\RN,\lambda)$, where $\lambda$ is the Lebesgue measure on $\mb R$, and  of course it is not a continuous frame. On the contrary, the definition of distribution frame reduces to that of discrete frame when $(X,\mu)$ is the set $\mathbb{N}$ with the counting measure $\gamma$.  
Indeed we have the following.

\begin{prop}
	\label{prop_distr_discr}
	Let $\omega: n\in \mb N\to \omega_n \in \mathcal D^\times$ be a bounded Bessel distribution (resp., distribution frame) on $\D [t]\subset \H \subset \D^\times[t^\times]$. Then $\{\omega_n\}_{n\in \mb N}\subset \H$ and $\{\omega_n\}_{n\in \mb N}$ is a Bessel sequence (resp., frame) of $\H$.
\end{prop}
\begin{proof}
	Assume that $\{\omega_n\}_{n\in \mb N}$ is a bounded Bessel distribution map and fix $m\in \mb N$. The linear functional $f\mapsto \ip{f}{\omega_m}$ for $f\in \D$ is bounded with respect to the norm of $\H$, because
	\begin{equation}
	\label{eq_1}
	|\ip{f}{\omega_m}|^2\leq \sum_{n\in\mb N} |\ip{f}{\omega_n}|^2\leq B \|f\|^2, \qquad \forall f\in \D.
	\end{equation}
	This means that $\omega_m\in \H$. A standard argument ({\cite[Lemma 5.1.9]{christensen2}}) shows that \eqref{eq_1} extends for each $f\in \H$, i.e. $\{\omega_n\}_{n\in \mb N}$ is a Bessel sequence. \\
	If $\{\omega_n\}_{n\in \mb N}$ is a distribution frame, then the conclusion follows in a similar way.
\end{proof}

\begin{defn}
	Let $\omega,\theta$ be distribution frames. We say that $\theta$ is a dual frame of $\omega$ if
	$$ \ip{f}{g}= \int_X\ip{f}{\theta_x}\ip{\omega_x}{g}d\mu, \quad \forall f,g \in \D.$$
\end{defn}

To formulate the following result, we recall that there exists a unique operator $R_\omega\in {\mc L}(\D)$ such that $S_\omega R_\omega f=f$ for every $f\in \D$ (\cite[Lemma 3.8]{TTT}).

\begin{prop}[{\cite[Proposition 3.10]{TTT}}] \label{prop_dual2} Let $\omega$ be a distribution frame with frame bounds 	$A$ and $B$.  Then the map $\theta:X\to \D^\times$ defined by $\theta_x:=R_\omega^\times\omega_x$ for $x\in X$, is a distribution frame with bounds $B^{-1}$ and $ A^{-1}$ and it is a dual frame of $\omega$.
\end{prop}
The map $\theta$ in Proposition \ref{prop_dual2} is called the {\em canonical dual frame} of $\omega$.

\begin{defn}[{\cite[Definition 2.3]{TTT}}]
	\label{def_dist_basis}
 Let $\D[t]$ be a locally convex space, $\D^\times$ its conjugate dual and $\omega: x\in X\to \omega_x \in \D^\times$ a weakly measurable map. Then
$\omega$ is a {\it distribution basis} for $\D$ if, for every $f\in\D$, there exists a {\it unique} $\mu$-measurable function $\xi_f$ such that:
$$
\ip{f}{g}=\int_X\xi_f(x)\ip{\omega_x}{g}d\mu, \quad\forall f,g\in\D
$$
and, for every $x\in X$, the linear functional $f\in\D\rightarrow\xi_f(x)\in\mathbb C$ is continuous in $\D[t]$.
\end{defn}
\noindent
Given a distribution basis $\omega$ we can simply write in weak sense
$$
f=\int_X\xi_f(x)\omega_x d\mu, \quad \forall f\in \D.
$$
Moreover, $\omega$ is $\mu$-independent. Since $f\in\D\rightarrow\xi_f(x)$ continuously, there exists a unique weakly $\mu$-measurable map $\theta:X\rightarrow \D^\times$ such that: $\xi_f(x)=\ip{f}{\theta_x}$ for every $f\in\D$. We call $\theta$ {\it dual} map of $\omega$. If $\theta$ is $\mu$-independent, then it is a distribution basis too.

The next two notions are the counterparts of orthonormal and Riesz bases of the discrete context, which are particular cases of Definition \ref{def_dist_basis}.	

\begin{defn}[{\cite{TTT}}] 
A weakly measurable map $\omega: X \to \D^\times$  is a {\em Riesz distribution basis} if it is a $\mu$-independent distribution frame.\\
A weakly measurable map $\zeta:X\to \D^\times$ is {\em Gel'fand distribution basis} if it is a $\mu$-independent Parseval distribution frame.
\end{defn}


In a way similar to the discrete case, we can give some equivalent conditions for a Bessel distribution map to be a Riesz distribution basis.

\begin{prop}[{\cite[Proposition 3.19]{TTT}}]\label{prop_rieszbasis}Let  $\D\subset\H\subset \D^\times$ be a rigged Hilbert space and let $\omega: x\in X \to \omega_x\in \D^\times$ be a Bessel distribution map. Then the following statements are equivalent.
\begin{itemize}
\item[(a)]  $\omega$ is a Riesz distribution basis;
\item[(b)] if $\zeta$ is a Gel'fand distribution basis, then the operator $W$ defined, for $f\in \H$, by
$$ f=\int_X \xi_f(x)\zeta_x d\mu \to W f= \int_X \xi_f(x)\omega_x d\mu$$ is continuous and has bounded inverse;
\item[(c)] the synthesis operator $T_\omega$ is a topological isomorphism of $L^2(X, \mu)$ onto $\H$;
\item[(d)] { $\omega$ is total and there exist $A,B>0$ such that
\begin{equation*}\label{eqn_ref} A\|\xi\|_2^2 \leq \left\|\int_X \xi(x)\omega_x d\mu \right\|^2 \leq B\|\xi\|_2^2, \quad \forall \xi \in L^2(X, \mu).\end{equation*}}
\end{itemize}
\end{prop}


If $\omega$ is a Riesz distribution basis with frame bounds $A,B$, then $\omega$ possesses a unique dual frame $\theta$ (so the canonical dual frame) which is also a Riesz distribution basis with frame bounds ${B^{-1}}$ and ${A}^{-1}$ (see \cite[Proposition 3.20]{TTT}). In particular, a Gel'fand distribution basis $\zeta$ coincides with its dual basis.

\begin{example}[{\cite[Example 3.18]{TTT}}] \label{ex_delta2} Let us come back to Example \ref{ex_delta}. The distribution frame $\delta: x\in {\mb R}\mapsto \delta_x\in \SSS$ is clearly $\lambda$-independent, then $\delta$ is a Gel'fand distribution frame.
\end{example}

	
	In Proposition \ref{prop_distr_discr} we made a consideration about the case $(X,\mu)=(\mathbb{N},\gamma)$ and discrete frames.
	Now we compare distribution frames with continuous frames. There is indeed a remarkable difference about the possibility to define `Riesz maps'. To explain the difference in details, we recall that a family $\{F_x\}_{x\in X}\subset \H$ is a Riesz continuous map if one of the following statements holds (see \cite{akrt,gabhan}) 
	\begin{itemize}
		\item[(i)] $\{F_x\}_{x\in X}$ is a continuous frame and the operator $C_F:\H \to L^2(X,\mu)$ defined by $(C_Ff)(x)=\ip{f}{F_x}$ is surjective;
		\item[(ii)] $\{F_x\}_{x\in X}$ is a continuous frame and  $\mu$-linearly independent, i.e. if $c:X\to \mathbb{C}$ and
		$
		\int_X c(x)\ip{f}{F_x}d\mu(x)=0$ for all $f\in \H$, then $c(x)=0$ $\mu$-almost everywhere.
	\end{itemize}
	It was proven in \cite[Corollary 4.3]{lll} and in \cite[Theorem 9]{bal_ms_2}, that Riesz continuous maps can be defined only if the space $(X,\mu)$ is atomic. In contrast, Riesz distribution maps can be defined with non-atomic measure spaces $(X,\mu)$ (as we have seen in Example \ref{ex_delta}). Note that the formulations of Riesz continuous map and Riesz distribution map are totally analogue.

\section{Distribution multipliers}\label{sect_3}

Let $\D[t]\subset \H \subset \D^\times[t^\times]$ be a rigged Hilbert space. In the distribution context, there is more than one way to define multipliers. The first way we describe consists of operators acting on $\D[t]$ with values in $\D^\times[t^\times]$.

We suppose that $\D[t]$ is a reflexive Fr\'echet space and  $(X,\mu)$ is a measure space with $\mu$ a $\sigma$-finite positive measure. Let $\omega,\theta:X\to\D^\times $ be two weakly measurable Bessel maps and $m\in L^\infty(X,\mu)$.
Then the sesquilinear form
$$
\Omega_{m,\omega,\theta}(f,g):=\int_X m(x)\ip{f}{\omega_x} \ip{\theta_x}{g} d\mu
$$
is defined for all $f,g\in\D$.
By Proposition \ref{prop2}(ii) we have
$$
|\Omega_{m,\omega,\theta}(f,g)|\leq \|m\|_\infty\|\ip{f}{\omega_x} \|_2\|\ip{\theta_x}{g}\|_2\leq \|m\|_\infty p(f)p(g)
$$
for all $f,g\in\D$. This means that $\Omega_{m,\omega,\theta}$ is jointly continuous on $\D[t]$ and then there exists an operator $\mathcal{M}_{m,\omega,\theta}\in\LDD$, such that
$$
\ip{\mathcal{M}_{m,\omega,\theta}f}{g}=\Omega_{m,\omega,\theta}(f,g), \qquad \forall f,g\in \D.
$$
For brevity we write $$
\mathcal{M}_{m,\omega,\theta} f= \int_X m(x)\ip{f}{\omega_x} {\theta_x} d\mu, \qquad \forall f\in \D
$$
and we call $\mathcal{M}_{m,\omega,\theta}$ the {\it outer} distribution multiplier of $\omega$ and $\theta$ with symbol $m$.

As in \cite{Balazs_basic_mult}, we have, if $T^\times_\omega$ and $T_\theta$ are the analysis and synthesis operators of $\omega$ and $\theta$, respectively, $D_m: L^2(X,\mu)\rightarrow L^2(X,\mu)$ is the multiplication by $m$ defined by $D_mf(x):=m(x)f(x)$, then  $\mathcal{M}_{m,\omega,\theta}=T_\theta D_m T^\times_\omega$. Moreover,  $\mathcal{M}_{m,\omega,\theta}^\dagger=\mathcal{M}_{\overline{m},\theta,\omega}$.

Now we move to the second way to define multipliers, namely operators acting on $\H$.
Again we consider a rigged Hilbert space $\D[t]\subset \H \subset \D^\times[t^\times]$, a measure space $(X,\mu)$ with $\mu$ a $\sigma$-finite positive measure, but the choice of $\omega,\theta$ and $m$ is more general. Indeed, let $\omega$, $\theta$ be two weakly measurable maps and $m:X\to \mb C$ a $\mu$-measurable function. Let us define the subspace  $
D(M_{m,\omega,\theta})$ of $f\in\D$ such that the integral
$$
\int_X m(x)\ip{f}{\omega_x} \ip{\theta_x}{g} d\mu
$$
is convergent for all $g\in\D$ and the linear functional
\begin{equation}
\label{funct}
g\mapsto \int_X m(x)\ip{f}{\omega_x} \ip{\theta_x}{g}d\mu
\end{equation}
is bounded with respect to the  norm of $\H$. An operator $M_{m,\omega,\theta}:D(M_{m,\omega,\theta})\rightarrow\H$ can be defined as follows: for every $f\in D(M_{m,\omega,\theta})$, $M_{m,\omega,\theta}f$ is the unique element in $\H$ associated to the functional \eqref{funct} by the Riesz lemma, i.e.
$$
\ip{M_{m,\omega,\theta} f}{g}=\int_X m(x)\ip{f}{\omega_x} \ip{\theta_x}{g} d\mu, \qquad \forall f\in D(M_{m,\omega,\theta} ),g\in \D.
$$
For shortness, we write $M_{m,\omega,\theta} f= \int_X m(x)\ip{f}{\omega_x} \theta_x$ and call $M_{m,\omega,\theta}$ the {\it distribution multiplier} of $\omega$ and $\theta$ with {\it symbol} $m$. If $\omega$ and $\theta$ are Bessel distribution maps (resp. Gel'fand bases, Riesz distribution bases, distribution frames), then $M_{m,\omega,\theta}$ is called a {\it Bessel distribution} (resp. {\it Gel'fand distribution, Riesz distribution, distribution frame}) multiplier.

In the language of representation of sesquilinear forms \cite{Kato,Schm}, we can say that $M_{m,\omega,\theta}$ is the operator associated to $\Omega_{m,\omega,\theta}$. In the discrete setting, operators associated to sesquilinear forms induced by sequences have been studied in \cite{Bag_sesq,Corso_seq,Corso_seq2}.

We first pay attention to distribution multipliers defined (resp. bounded) on $\D$.

\begin{prop}
	\label{Bessel_mult}
	Let $m\in L^\infty(X,\mu)$.
	\begin{enumerate}
		\item[(i)] If  $\omega$ is a Bessel distribution map and $\theta$ is a bounded Bessel distribution map, then $M_{m,\omega,\theta}$ is a well-defined operator $M_{m,\omega,\theta}:\D\to \H$.
		\item[(ii)] If  $\omega$ and $\theta$ are bounded Bessel distribution maps with bound $B_\omega, B_\theta$, respectively, then  $M_{m,\omega,\theta}$ is bounded and it extends to a bounded operator $\widehat{M}_{m,\omega,\theta}$ on $\H$ with norm $\|\widehat{M}_{m,\omega,\theta}\|\leq \sqrt{B_\omega B_\theta}\|m\|_\infty$.
		\item[(iii)] If $\omega,\theta$ are bounded Bessel maps, then $\widehat{M}_{m,\omega,\theta}=T_\theta D_m T^*_\omega$ and ${M_{m,\omega,\theta}}^*=\widehat{M}_{\overline{m},\omega,\theta}$.
	\end{enumerate}
\end{prop}

The proof is an adaptation of \cite[Lemma 3.3]{balaszetal} and it is omitted. However, we want to make the following remark about Proposition \ref{Bessel_mult}.

\begin{rem}
	Let $m\in L^\infty(X,\mu)$ and $\omega$, $\theta$ Bessel distribution maps, then $M_{m,\omega,\theta}$ is not necessarily bounded in the norm of $\H$. Indeed  
		let us consider $\SC \subset L^2(\mb R)\subset \SSS$ and
		$\omega: \mb R \to  \SSS$ defined by $\omega_x=x\delta_x$, i.e. the distributions
		$\ip{f}{x\delta_x}=xf(x)$ for $f\in \SC$. Then $\omega$ is a Bessel distribution map since
		$
		\int_\mb R |\ip{f}{x\delta_x}|^2dx = \int_\mb R |xf(x)|^2dx
		$
		is finite for all $f\in \SC$. Let $\theta$ be the distribution frame given by $\theta_x=\delta_x$ and $m(x)=1$ for $x\in \mathbb{R}$, then $M_{m,\omega,\theta}$ is defined on $f\in \SC$ and $(M_{m,\omega,\theta} f)(x) =xf(x)$ for $x\in \mb R$. Clearly,  $M_{m,\omega,\theta}$ is not bounded.
\end{rem}

Multipliers with a bounded inverse defined on the whole space have a special interest, since they lead to reconstruction formulas, as shown in the discrete case in \cite{bal_invers,bal_dual_frame,bal_dual_frame2}. In the following result we show how reconstruction formulas can be found by a distribution multiplier having a right or left inverse. 

\begin{thm}
	\label{prop_33}
	Let $\omega,\theta:X\to \D^\times$ be weakly measurable maps and $m:X\to \mathbb{C}$ such that the distribution multiplier  $M_{m,\omega,\theta}$ is defined on $\D$.
	\begin{enumerate}
		\item[(i)] If there exists $J\in \mathcal{L}(\D)$ such that $M_{m,\omega,\theta} J f=f$ for every $f\in \D$, then  the weakly measurable map $\rho:X\to \H$ defined by $\rho_x=J^\dagger (\overline{m(x)}\omega_x)$ satisfies
		$$
		\ip{f}{g}=\int_X \ip{f}{\rho_x}\ip{\theta_x}{g}d\mu, \qquad\forall f,g\in \D.
		$$
		\item[(ii)] If there exists $K\in \mathcal{L}(\D^\times)$ such that $K M_{m,\omega,\theta} f=f$ for every $f\in \D$
		the weakly measurable map $\tau:X\to \H$ defined by $\tau_x=K ({m(x)}\theta_x)$ satisfies
		$$
		\ip{f}{g}=\int_X \ip{f}{\omega_x}\ip{\tau_x}{g}d\mu, \qquad\forall f,g\in \D.
		$$
	\end{enumerate}
\end{thm}
\begin{proof}
	(i) Let $J\in \mathcal{L}(\D)$ as in the statement. Then
	\begin{align*}
	\ip{f}{g}=\ip{M_{m,\omega,\theta} J f}{g}=\int_X m(x)\ip{Jf}{\omega_x}\ip{\theta_x}{g} d\mu=\int_X \ip{f}{J^\dagger (\overline{m(x)}\omega_x)}\ip{\theta_x}{g}d\mu.
	\end{align*}
	
	(ii) Let $K\in \mathcal{L}(\D^\times)$ as in the statement. Then
	\begin{align*}
	\ip{f}{g}&=\ip{KM_{m,\omega,\theta} f}{g}=\ip{M_{m,\omega,\theta} f}{K^\dagger g}=\int_X m(x)\ip{f}{\omega_x}\ip{\theta_x}{K^\dagger g} d\mu\\
	&=\int_X \ip{f}{\omega_x}\ip{K ({m(x)}\theta_x)}{g}d\mu. \qedhere
	\end{align*}
\end{proof}

	\begin{example}
		Let us consider again the rigged Hilbert space $\SC \subset L^2(\mb R) \subset \SSS$ and $\omega=\theta$ the distribution frames defined by $\omega_x=\theta_x=\delta_x$ for every $x\in \mb R$ and $m\in C^\infty(\mb R)$ a function such that $0<\inf_{x\in \mb R}|m(x)| \leq \sup_{x\in \mb R}|m(x)|<\infty $. The  multiplier  $M_{m,\omega,\theta}$ is of course defined on $\SC$ and $M_{m,\omega,\theta} f=mf$. Clearly, the operator $J:\SC \to \SC$ defined by $Jf=m^{-1}f$ belongs to $\mathcal{L}(\SC)$ and the operator $K:\SSS \to \SSS$ defined by $KF=m^{-1}F$ belongs to $\mathcal{L}(\SSS)$. Moreover, $KM_{m,\omega,\theta}  f=M_{m,\omega,\theta} J f=f$ for every $f\in \SC$.  The reconstruction formulas in Theorem \ref{prop_33} hold in particular with $\rho_x=\tau_x =\delta_x$.
	\end{example}

In Section \ref{sect_4} we will give conditions for a Riesz multiplier to be invertible with bounded inverse. 



\section{Unbounded distribution multipliers}\label{sect_4}

In Proposition \ref{Bessel_mult} we gave a condition for a  distribution multiplier to be bounded. Not only bounded multipliers are interesting, of course; we refer to \cite{bit_jmp,Bag_sesq} where unbounded discrete multipliers have been studied in the context of non-selfadjoint hamiltonians. So in this section we want to analyze some aspects of unbounded distribution multipliers.

	Assuming that $\D[t]$ is a reflexive  Fr\'echet space,  $\omega,\theta$ are Bessel distribution maps and $m\in L^\infty(X,\mu)$, we easily see that  $M_{m,\omega,\theta}$ is actually a restriction of $\mathcal{M}_{m,\omega,\theta}$. Indeed,  $D(M_{m,\omega,\theta})=\{f\in \D: \mathcal{M}_{m,\omega,\theta}f\in \H\}$ and $M_{m,\omega,\theta}={\mathcal{M}_{m,\omega,\theta}}_{|D(M_{m,\omega,\theta})}$. This fact leads, for instance, to the following conclusions.
	
	\begin{prop}
		\label{pro_41}
		Let $\D[t]$ be a reflexive  Fr\'echet space. Let $\omega,\theta$ be Bessel distribution maps and $m\in L^\infty(X,\mu)$. If  $\mathcal{M}_{m,\omega,\theta}:\D\to \D^\times $ is bijective with a bounded inverse, then
		\begin{enumerate}
			\item[(i)] $M_{m,\omega,\theta}:D(M_{m,\omega,\theta})\to \H$ is bijective, densely defined and has a bounded inverse  (consequently $M_{m,\omega,\theta}$ is closed);
			\item[(ii)] ${M_{m,\omega,\theta}}^*=M_{\overline{m},\theta,\omega}$.
		\end{enumerate}
		
	\end{prop}
	\begin{proof}
		\begin{enumerate}
			\item[(i)] 	That  $M_{m,\omega,\theta}:D(M_{m,\omega,\theta})\to \H$ is bijective follows easily since $M_{m,\omega,\theta}$ is a restriction of $\mathcal{M}_{m,\omega,\theta}$. Since the inclusions 
			$\D[t] \subset  \H \subset\D^\times[t^\times]$ are continuous
			\begin{itemize}
				\item  there exists a continuous seminorm $p$ on $\D[t]$ and $\|f\|\leq   p(f)$ for all $f\in \D$;
				\item for all continuous seminorms $q$ on $\D^\times[t^\times]$ there exists $\alpha_q>0$ such that we have $  q(f)\leq \alpha_q \|f\|$ for all $f\in \D$.
			\end{itemize}
			By hypothesis ${\mathcal{M}_{m,\omega,\theta}}^{-1}:\D^\times\to \D$ is bounded, so there exists a continuous semi-norm $q$ on $\D^\times[t^\times]$ such that for all continuous semi-norms $p$ on $\D[t]$ we have $p({M_{m,\omega,\theta}}^{-1} F)\leq q(F)$ for all $F\in \D^\times[t^\times]$.
			Hence, for all $h\in \H$
			$$
			\|{M_{m,\omega,\theta}}^{-1}h\|\leq  p({M_{m,\omega,\theta}}^{-1}h)=   p({\mathcal{M}_{m,\omega,\theta}}^{-1}h) \leq  q(h) \leq \alpha_q \|h\|.
			$$
			Thus $M_{m,\omega,\theta}$ has a bounded inverse. The continuity of ${\mathcal{M}_{m,\omega,\theta}}^{-1}$ implies also that $D(M_{m,\omega,\theta})$ is dense, because  $D(M_{m,\omega,\theta})$ is the inverse image of $\H$ which is dense in $\D^\times[t^\times]$.
			\item[(ii)] Clearly, ${M_{m,\omega,\theta}}^*$ is bijective and $M_{\overline{m},\theta,\omega}\subseteq {M_{m,\omega,\theta}}^*$. The conclusion of point (i) holds also for $M_{\overline{m},\theta,\omega}$ since $M_{\overline{m},\theta,\omega}$ coincides with a restriction of $\mathcal{M}_{m,\omega,\theta}^\dagger=\mathcal{M}_{\overline{m},\theta,\omega}$. Thus we can conclude that $M_{\overline{m},\theta,\omega}= {M_{m,\omega,\theta}}^*$. \qedhere
		\end{enumerate}		
	\end{proof}

Even though $M_{m,\omega,\theta}$ is not necessarily bounded, Proposition \ref{pro_41} makes use of boundedness of $m$. This may be a strong hypothesis, thus we now look for less restrictive assumptions to ensure that $M_{m,\omega,\theta}$ is densely defined.

In the discrete context it is very easy to prove that a multiplier $M_{m,\phi,\psi}$ of a Hilbert space $\H$, where $\phi=\{\phi_n\}_{n\in \mb N}$ is a Riesz basis and $\psi=\{\psi_n\}_{n\in \mb N}$ a sequence of $\H$, is densely defined whatever the symbol $m=\{m_n\}_{n\in \mb N}$ is. Indeed, there exist a unique total sequence (in particular a Riesz basis) $\widetilde \phi=\{\widetilde{\phi}_n\}$ biorthogonal to $\phi$, i.e. $\ip{\widetilde{\phi}_m}{\phi_n}=\delta_{m,n}$ (the Kronecker symbol) for all $m,n\in \mb N$.
Thus $D(M_{m,\phi,\psi})$ is dense, because it contains $\widetilde \phi$.

On the contrary, when $\D\subset \H \subset \D^\times$ is a rigged Hilbert space and $\omega:X\to \D^\times$ is a Riesz distribution basis, then we may not find a function $\rho:X\to \D$ which is biorthogonal to $\omega$ in the sense that
$$\ip{\rho_y}{\omega_x}=\begin{cases}
1 \quad \text{ if } x=y,\\
0 \quad \text{ if } x\neq y.
\end{cases}
$$
Indeed,	let us consider the Riesz distribution basis given by the Dirac deltas $\omega_x=\delta_x$, $x\in \mb R$, on the rigged Hilbert space $\SC \subset L^2(\mb R) \subset \SC^\times $ (Example \ref{ex_delta}). Then there is no $x\in \mb R$ and $f\in \SC $ such that $f(x):=\ip{f}{\delta_x}=1$ and $f(z):=\ip{f}{\delta_z}=0$ for all $z\neq x$. Thus we have to manage a new problem in order to study densely defined distribution multipliers.

	Taking again the example of $\omega_x=\delta_x$, $x\in \mb R$, on  $\SC \subset L^2(\mb R) \subset \SC^\times $, we note that for any symbol $m: \mb R \to \mb C$ and for $\theta=\omega$ the multiplier $M_{m,\omega,\theta}$ is densely defined. We will give the proof in Theorem \ref{th_hyper_pseudo} in a more general context. Here we confine ourselves to give the following remark. Note that we say that a subset $V$ of a Hilbert space $\H$ is {\it total} if $\ip{f}{h}=0$ for all $f\in V$ implies $h=0$ (then the linear span of $V$ is dense in $\H$).
	
	\begin{rem}
		\label{rem_hyper}
		Let $\lambda$ the Lebesgue measure on $\mb R$. Let $\alpha:\mb R\to \mb C$ be a positive $\lambda$-measurable function and define $V_\alpha:=\{f\in C_0^\infty (\mb R): |f(x)|\leq \alpha(x), x\in \mb R\}$. We prove that the subset $V_\alpha$ is total in $L^2(\mb R)$ dividing the proof into three steps.
		
		First of all, $\alpha$ is locally bounded away from zero in a.e. $x\in \mb R$, i.e. $B^c:=\mb R\backslash B$ is measurable with measure zero, where
		$$B=\{x\in \mb R:\text{there exists an interval $U_x\subset \mb R$ of $x$ such that essinf$_{y\in U_x} \alpha(y)>0$}\}.$$
		Indeed, fix $n\in \mb N$. For every $x\in B^c$, there exists a measurable set $U_{n,x}$ containing $x$ such that essinf$_{y\in U_{n,x}} \alpha(y)\leq\frac{1}{n}$. Then $B^c\subseteq \cup_{x\in B^c} U_{n,x}\subseteq \{y\in \mb R:\alpha(y) \leq \frac{1}{n}\}$, thus the outer measure $\lambda_o(B^c)\leq \lambda(\{y\in \mb R:\alpha(y) \leq \frac{1}{n}\})\to 0$ for $n\to \infty$.
		Hence $B^c$ is measurable with measure zero.

		Now, if $x\in B$ and $m_x=\;$essinf$_{y\in U_x} \alpha(y)>0$, then $m_x\chi_{I}\in V_\alpha$ where $\chi_{I}$ is the characteristic function of any interval $I\subset U_x$. Taking into account that $B^c$ has measure zero, we conclude that  there is a subset $V$ total in $L^2(\mb R)$ such that every $f\in V$ satisfies $f(x)=m\chi_{I}$ (for some $m>0$ and an interval $I$) and $0\leq f(x)\leq \alpha(x)$ for every $x\in \mb R$.	
		
		Finally, every function $f\in V$ can be approximated with $C_0^\infty$ functions $\{f_k\}$ with $0\leq f_k\leq f$. Hence, $V_\alpha$ is total.
	\end{rem}
	
	The reason of talking about the example $\omega_x=\delta_x$, is that it suggests to consider variations of the condition of biorthogonality as stated in the next definitions.
	For a better comparison, we rewrite a property of $V_\alpha$ (with $\alpha:\mb R\to \mb C$ a positive measurable function) in the following way: for every $f\in V_\alpha$ there exists a bounded subset $X_f\subset \mb R$ and $|\ip{f}{\delta_x}|\leq \alpha(x)$ for $x\in X_f$ and $\ip{f}{\delta_x}=0$ for $x\notin X_f$. 

\begin{defn}
	\label{def_typeAB}
	Let $\D\subset \H \subset \D^\times$ be a rigged Hilbert space and $\omega:X\to \D^\times$ a weakly measurable	function. We say that
	\begin{enumerate}
		\item[(i)]  $\omega$ is {\it pseudo-orthogonal} if  there exists a  subset $V\subset \D$ total in $\H$ such that for every $f\in V$ there exists a measurable subset $X_f\subset  X$ with $\mu(X_f)<\infty$, $\sup_{x\in X_f}|\ip{f}{\omega_x}|<\infty$ for $x\in X_f$ and $\ip{f}{\omega_x}=0$ for $x\notin X_f$;
		\item[(ii)]  $\omega$ is {\it hyper-orthogonal} if for every positive measurable function $\alpha:X\to \mb C$ there exists a subset $V_\alpha\subset \D$ total in $\H$ such that for every $f\in V_\alpha$ there exists a measurable subset $X_f\subset  X$ with $\mu(X_f)<\infty$, $|\ip{f}{\omega_x}|\leq \alpha(x)$ for $x\in X_f$ and $\ip{f}{\omega_x}=0$ for $x\notin X_f$.
	\end{enumerate}
\end{defn}

Note that these definitions are covered by a Riesz (discrete) basis $\{\phi_n\}_{n\in \mb N}$ (more generally by a sequence $\{\phi_n\}_{n\in \mb N}$  having a total biorthogonal sequence $\{\psi_n\}_{n\in \mb N}$, i.e. $\ip{\phi_n}{\psi_m}=\delta_{n,m}$). Furthermore,
if $\omega$ is hyper-orthogonal then it is also pseudo-orthogonal. We are now able to formulate results about the density of domains of distribution multipliers. We denote by $ L^2_{loc}(X,\mu)$ the space of measurable functions $f$ on $X$ such that $f\in L^2(U)$ for every bounded measurable subset $U\subseteq X$.

\begin{thm}
	\label{th_hyper_pseudo}
	Let $\D[t]\subset \H \subset \D^\times[t^\times]$ be a rigged Hilbert space, $\omega:X\to \D^\times$ a weakly measurable	function, $\theta:X\to \D^\times$ a bounded Bessel distribution map with  Bessel bound $B_\theta$ and $m:X\to \mb C$ a $\mu$-measurable function.
	\begin{enumerate}
		\item[(i)] If $\omega$ is pseudo-orthogonal and $m\in L^2_{loc}(X,\mu)$, then the distribution multiplier $M_{m,\omega,\theta}$ is densely defined.
		\item[(ii)] If $\omega$ is a bounded Bessel distribution map and hyper-orthogonal, then the distribution multiplier $M_{m,\omega,\theta}$ is densely defined.
	\end{enumerate}
\end{thm}
\begin{proof}
	\begin{enumerate}
		\item[(i)] Let $V$ and $X_f$ be as in Definition \ref{def_typeAB}(1). For $f\in V, g\in \D$ we have by Cauchy-Schwarz inequality that
		$$
		\left |\int_X m(x)\ip{f}{\omega_x} \ip{\theta_x}{g}d\mu \right|\leq C	\int_{X_f} |m(x)||\!\ip{\theta_x}{g}\!|d\mu \leq CB_\theta^{\frac{1}{2}}\|m\|_{L^2(X_f)}  \|g\|,
		$$
		where $C=\sup_{x\in X_f}|\ip{f}{\omega_x}|$. Thus $f\in D(M_{m,\omega,\theta})$, and consequently $D(M_{m,\omega,\theta})$ is dense because $V$ is total.
		\item[(ii)] The proof is divided into three parts.  If $m\in L^\infty(X,\mu)$, then the conclusion follows by Proposition \ref{Bessel_mult} since  $D(M_{m,\omega,\theta})=\D$. If $|m(x)|\geq 1$ a.e., then we take $\alpha(x)=|m(x)|^{-1}$ a.e. in Definition \ref{def_typeAB}(ii). Thus there exists $V_\alpha\subset \D$  total in $\H$ such that for every $f\in V_\alpha$ there exists a compact subset $X_f\subset X$ and
		$$
		\left |\int_X m(x)\ip{f}{\omega_x} \ip{\theta_x}{g}d\mu \right|\leq	\int_{X_f} |m(x)||m(x)|^{-1}|\ip{\theta_x}{g}|d\mu \leq \mu(X_f)^{\frac{1}{2}}B_\theta^{\frac{1}{2}} \|g\|,
		$$
		for every $g\in \D$ (the last inequality is due to Cauchy-Schwarz inequality). This means that $D(M_{m,\omega,\theta})$ is dense. Finally, let $m$ be a generic measurable function. Then it is possible to write $m$ as sum of two measurable functions $m=m_1+m_2$ such that $m_1\in L^\infty(X)$ and $|m_2|\geq 1$. Then $M_{m_1,\omega,\theta}$ is well-defined on $\D$ and $M_{m_2,\omega,\theta}$ is defined on a  subspace of $\D$ dense in $\H$. As consequence, also $M_{m,\omega,\theta}=M_{m_1,\omega,\theta}+M_{m_2,\omega,\theta}$ is densely defined. \qedhere
	\end{enumerate}
\end{proof}

We show other examples of weakly measurable maps satisfying Definition \ref{def_typeAB}.

	\begin{example}
		For $x\in \mb R$ consider the function $\omega_x$ defined by $\omega_x(y):=e^{-2\pi i x y}$ for $y\in \mb R$. Then $\omega_x$ is a distribution on $L^1(\mb R)\cap L^2(\mb R)$ and in Example 3.17 of \cite{TTT} it was proved that $\omega$ is a distribution frame.
		Choosing $V=\{f\in L^2(\mb R): \widehat{f}\in C^\infty_0(\mb R)\}$, where $\widehat{f}$ denotes the Fourier transform of $f$, in Definition \ref{def_typeAB}, we conclude that $\omega$ is pseudo-orthogonal. \\
		Now let $\alpha:\mb R\to \mb C$ be a positive measurable function and define $V_\alpha=\{f\in L^2(\mb R): \widehat{f}\in C^\infty_0(\mb R) \text{ and } |\widehat f(x)| \leq \alpha(x), x\in \mb R\}$. By the considerations in Remark \ref{rem_hyper} and since the Fourier transform is unitary in $L^2(\mb R)$, the set $V_\alpha$ is  total, i.e. $\omega$ is also hyper-orthogonal.
	\end{example}

	\begin{example}
		Let $\D[t]$ be a dense subspace of $L^2(\mb R)$ endowed with a locally convex topology $t$, stronger than the topology of $L^2(\mb R)$, and such that $C_0^\infty(\mb R)\subset \D$. \\
		Let $g\in L^2(\mb R)$ have support in a bounded interval $I$.  Define $\omega_x(t)=g(t-x)$ for $t\in \mb R$. Then the weakly measurable map $\omega:\mb R \to L^2(\mb R)$ is pseudo-orthogonal (again one can take $V=C_0^\infty(\mb R)$). 	
	\end{example}

	We conclude this section by turning the attention to a sufficient condition for a distribution multiplier to be closable.
	
	\begin{prop}
		Let $\omega,\theta:X\to \D^\times$ be weakly measurable	functions and $m:X\to \mb C$ be a $\mu$-measurable function. If $M_{\overline{m},\theta,\omega}$ is densely defined (in particular if $\omega,\theta$  are bounded Bessel distribution maps and $\theta$ is hyper-orthogonal), then $M_{m,\omega,\theta}$ is closable.
	\end{prop}
	\begin{proof}
		For $f\in D(M_{m,\omega,\theta})$ and $g\in D(M_{\overline{m},\theta,\omega})$ we have $\ip{M_{m,\omega,\theta} f}{g}=\ip{f}{M_{\overline{m},\theta,\omega} g}$. This means that  $M_{m,\omega,\theta}\subseteq {(M_{\overline{m},\theta,\omega})}^*$, i.e. $M_{m,\omega,\theta}$ is closable.	
	\end{proof}

\section{Riesz distribution multipliers}
\label{sect_5}
In this section, we examine the case where $\omega$ and $\theta$ are Riesz distribution bases,  reconsidering the Examples 4.1 and 4.2 of \cite{TTT}.
Let $\omega$ and $\theta$ be distribution Riesz bases, ${\check\omega}$ and ${\check\theta}$ their extensions to $\H$ with the limit procedure described just after Definition \ref{bbounded},  and  a $\mu$-measurable function $m:X\to \mb C$ such that the integral:
$$
\int_X m(x)\ip{f}{\check\omega_x}\ip{\check\theta_x}{g}d\mu
$$
is convergent for all $f,g\in\H$. Since $\theta$ is a Riesz basis, the analysis operator $T_\theta^* $ is a topological isomorphism of $\H$ onto $L^2(X,\mu)$, then $T_\theta^*(\H)=L^2(X,\mu)$. It follows that $m(x)\ip{f}{\check\omega_x}\in L^2(X,\mu)$ for all $f\in \H$. Furthermore, the operator $M_{m, \omega,\theta}:\D\rightarrow\H$:
$$
M_{{m}, \omega,\theta,} f=\int_X {m(x)}\ip{f}{\omega_x}\theta_x d\mu,\qquad \forall f\in\D,
$$
is well-defined. 
Analogously, for all $g\in\H$ one has $m(x)\ip{g
}{\check\theta_x}\in L^2(X,\mu)$ and the operator $M_{m, \omega,\theta}^\dag:\D\rightarrow \H$:
$$
M_{m, \omega,\theta}^\dag g:=M_{\overline{m},\theta, \omega} g=\int_X \overline{m(x)}\ip{g}{\theta_x}\omega_x d\mu\quad \forall g\in\D,
$$
is well-defined. One has: 
$$
\ip{M_{m, \omega,\theta}f}{g}=\ip{f}{M_{m, \omega,\theta}^\dag g},\quad \forall f,g\in\D.
$$
Then $M_{m, \omega,\theta}$ is a closable operator in $\H$. It is not difficult to show that the domain of the closure $D(\overline{M}_{m, \omega,\theta})$ is 
$\{ f\in\H: \int_X |m(x)\ip{f}{\check\omega_x}|^2d\mu<\infty\}$. In general, the operators $M_{m, \omega,\theta}$ are unbounded, so their product is not always defined. However, if they belong to the space $\LD$, they can be multiplied. In the following example, some cases of unbounded multipliers in $\LD$ are considered.
\begin{example}\label{ld}
	Let us consider the rigged Hilbert space $\SC \subset L^2(\mb R) \subset \SC^\times $.
	We write $\widehat f$ and $\check{f}$, respectively, for the Fourier transform and inverse Fourier transform of $f\in \SC$.	
	
	Define:		
	$\omega_x=\delta_x$ and $\theta_x(y):=e^{-2\pi i x y}$ for $y\in \mb R$, thus $\delta_x,\theta_x \in \SC^\times$. Let 
	$\mathcal O_M(\mb R)$ be the space of $C^\infty$-functions which, together with their derivatives, are polynomially bounded (see \cite{horvath}). If $m\in \mathcal O_M(\mb R)$,
	for $f\in \SC$ we have:
	$$M_{m,\omega,\omega}f=mf,\; M_{m,\omega,\theta}f=\check{m}\ast \check{f},\; M_{m,\theta,\omega}f=m\widehat f \;\text{ and }\; M_{m,\theta,\theta}f=\check{m}\ast f.
	$$
	
\end{example}
The above considerations lead in particular to the case of a Riesz basis $\omega$ and its dual $\theta$ of Example 4.2 \cite{TTT}. To simplify the notation, we denote the multiplier as $M_{m}:=M_{m,\theta,\omega}$. In Example 4.2 \cite{TTT}, it is shown that $m(x)$ is a generalized eigenvalue of $M_{m}$, i.e.:
$$\ip{(M^\dagger_{m})^\times\theta_x}{g}=m(x)\ip{\theta_x}{g}, \forall g\in\D,\quad \mbox{$\mu$-a.e.}\quad x\in X,$$
and ${\overline m}(x)$ is a generalized eigenvalue of $M^\dagger_{m}$ i.e.: 
$$\ip{(M_{m})^\times\omega_x}{g}=\overline{m}(x)\ip{\omega_x}{g} \forall g\in\D,\quad \mbox{$\mu$-a.e.}\quad x\in X.$$
A consequence is the following:
\begin{prop}\label{mult}
	Let  $M_{m_1}$ and $M_{m_2}$ be multipliers of a Riesz basis $\omega$ and its dual $\theta$, such that $M_{m_i}\in \LD $, $i=1,2$. Then $M_{m_1} M_{m_2}= M_{m_1m_2}$.
\end{prop}
\begin{proof}
	
	For all $f,g\in \D$
	$$
	\ip{(M_{m_2} M_{m_1})f}{g}=\ip{ M_{m_1}f}{M^\dagger_{m_2}g}=\int_X m_1(x)\ip{f}{\omega_x}\ip{\theta_x}{M^\dagger_{m_2}g}d\mu=
	$$
	$$
	=\int_X m_1(x)m_2(x)\ip{f}{\omega_x}\ip{\theta_x}{g}d\mu=\ip{M_{m_1m_2}f}{g}
	$$
	and the proof is completed.
\end{proof}
Analogously,  $M^\dag_{m_1}$ and $M^\dag_{m_2}$ can be multiplied, and $(M_{m_1}M_{m_2} )^\dag=M^\dag_{m_2} M^\dag_{m_1}$. This shows that the multipliers in $\LD$ of dual Riesz bases is a $\dag$-subalgebra in $\LD$. 
Furthermore, if both $M_m^{-1}$ and $M_{\frac{1}{m}}$ are defined in $\LD$, by Proposition \ref{mult} one has that $M_m^{-1}=M_{\frac{1}{m}}$. 
The considered case of multipliers in $\LD$  allows to handle easily the symbolic calculus, but operators in $\LD$ are, in general, unbounded. What can be said about the case of bounded operators? Obviously, if $m\in L^\infty(X,\mu)$ and $\omega$ and $\theta$ are Riesz bases, $M_m$ is bounded.  Vice versa, the following proposition holds:
\begin{prop}
	Let $\omega$ be a distribution Riesz basis with dual $\theta$. If the multiplier $M_{m,\omega,\theta}$ is bounded then $m\in L^\infty(X,\mu)$.
\end{prop}
\begin{proof}
	The proposition is true for the diagonal operator $A_{m,\zeta}$, i.e. a multiplier of the Gel'fand basis: $\omega=\theta=\zeta$ (see Example 4.1 of \cite{TTT}). The same holds for  $M_{m,\omega,\theta}$: in fact $M_{m,\omega,\theta}$ and $A_{m,\zeta}$ are similar  via $W$ of b) of Proposition \ref{prop_rieszbasis}, i.e. ${\overline M}_{m,\omega,\theta}=W\overline{A}_{m,\zeta} W^{-1}$ (see  Example 4.2 of \cite{TTT}), where ${\overline M}_{m,\omega,\theta}$ and $\overline{A}_{m,\zeta}$ are their closure.
\end{proof}

For conditions of invertibility of a multiplier, we can state  the following  in a more general form. 

\begin{prop}
	\label{invertibility}
	Let  $\omega,\theta:X\to \D^\times$  be weakly measurable maps and $m(x)\neq0$  $\mu$-a.e. in $X$. Then
	\begin{enumerate}
		\item[(i)] If $\omega$ is $\mu$-independent and $\theta$ is total, then ${M}_{m,\omega,\theta}$ is injective.
		\item[(ii)] If  $\omega$ is total and $\theta$ is $\mu$-independent, then ${M}_{m,\omega,\theta}$ has dense range in $\H$.
	\end{enumerate}		
\end{prop}
\begin{proof}
	Assume that ${M}_{m,\omega,\theta} f=\int_X m(x)\ip{f}{\omega_x}{\theta_x}d\mu=0$. Since $\theta$ is $\mu$-independent, we have $m(x)\ip{f}{\omega_x}=0$ a.e., that is $\ip{f}{\omega_x}=0$ a.e.; but ${\omega}$ is total, then $f=0$.
	For the range, let $g\in\D$ such that $\ip{{M}_{m,\omega,\theta}f}{g}=0$ for all $f\in\D$, that is  $\ip{{M}_{m,\omega,\theta} f}{g}=\int_X m(x)\ip{f}{\omega_x}\ip{\theta_x}{g}d\mu=0$. Since $\omega$ is $\mu$-independent, we have $m(x)\ip{\theta_x}{g}=0$ a.e. and $g=0$, because $\theta$ is total.
\end{proof}
In particular, if $\omega$, $\theta$ are Riesz bases and $m(x)$ is nonzero a.e., then $M_{m,\omega,\theta}$ is invertible with densely defined inverse. To have a bounded inverse we can make use of an additional assumption. 

 \begin{prop}
	Let $M_{m,\omega,\theta}$ be a Riesz  distribution  multiplier. If  there exists $C>0$ such that $0<C\leq |m(x)|$ for all $x\in X$, the inverse of $M_{m,\omega,\theta}$ is bounded.
\end{prop}
\begin{proof}
By Proposition \ref{invertibility}, the inverse exists. Let  $A_\omega,B_\omega$ and $A_\theta,B_\theta$ be the  lower and upper bounds of $\omega$ and
$\theta$, respectively. The operator $M_{m,\omega,\theta}$ has closure extension $\widehat M_{m,\omega,\theta}:=T_\theta D_m T_\omega^*$. By Proposition \ref{prop_rieszbasis}, the operators $T^*_\omega,T_\omega,T_\theta,T^*_\theta$ are bounded, invertible with bounded inverses, so we have:
	$$
	{A_\theta}{A_\omega}C\|f\|\leq {A_\theta}C\|T^*_\omega f\|_2\leq  \|T_\theta D_m T^*_\omega f\|,\quad\forall f\in D(M_{m,\omega,\theta}),
	$$
	and the proof is completed. 
\end{proof}

%
%
%

\section{Conclusions}
Some questions about symbolic calculus in a more general set-up (not only in the case of dual Riesz bases) are open. For instance, it is known that in $\LDD$ a partial multiplication is defined (see \cite{ait_book,trts}), thus a symbolic calculus may be developed for multipliers in $\LDD$. The idea behind Definition \ref{def_typeAB} are connected to localization frames which were introduced in \cite{gro} and further studied (sometimes with variations)  in \cite{bchl,bchl2,batamitk,forgro,forrau,gro2}. More precisely, some result in Section \ref{sect_4} can be extended considering a certain decay of $x\mapsto \ip{f}{\omega_x}$ instead of assume that $\ip{f}{\omega_x}$ is null outside a bounded set.

\section*{Acknowledgments}

This work has been supported by the ``Gruppo Nazionale per l'Analisi Matematica, la Probabilit\`{a} e le loro Applicazioni'' (GNAMPA-INdAM).

\bibliographystyle{amsplain}

\begin{thebibliography}{99}
	
	
	
	
	
\bibitem{AAG_paper} S. T. Ali, J. P. Antoine, J. P. Gazeau, {\it Continuous frames in Hilbert spaces}, Annals
of Physics, {\bf 222} (1993) 1-37.
\bibitem{ait_book} J.-P. Antoine, A. Inoue, C. Trapani, {\em Partial *-algebras
and their Operator Realizations}, Kluwer, Dordrecht, (2002).
 \bibitem{akrt}
 A. A. Arefijamaal, R. A. Kamyabi Gol, R. Raisi Tousi,  N. Tavallaei, {\it A new approach to continuous Riesz bases}, J. Sciences, Islamic Republic of Iran, 24(1):63–69, (2012).
\bibitem{bit_jmp}F. Bagarello, A. Inoue, C.Trapani, {\em Non-self-adjoint hamiltonians defined by Riesz bases}, J. Math. Phys. 55, 033501 (2014).
\bibitem{Bag_sesq}
F. Bagarello, H. Inoue, C. Trapani, {\it Biorthogonal vectors, sesquilinear forms, and some physical operators}, J. Math. Phys. 59, 033506, (2018).
\bibitem{bchl}
R. Balan, P.G. Casazza, C. Heil, Z. Landau, {\it Density, overcompleteness, and localization of frames I. Theory}, J. Fourier Anal. Appl. 12, 105–143 (2006)
\bibitem{bchl2}
R. Balan, P.G. Casazza, C. Heil, Z. Landau, {\it Density, overcompleteness, and localization of frames II. Gabor systems}, J. Fourier Anal. Appl. 12, 309–344 (2006)
\bibitem{Balazs_basic_mult}
P. Balazs, {\it Basic definition and properties of Bessel multipliers}, J. Math. Anal. Appl., 325(1), 571-585, (2007).
\bibitem{balaszetal}P. Balazs, D. Bayer, A. Rahimi, {\em Multipliers for continuous frames
in Hilbert spaces}, J. Phys. A: Math. Theor., {\bf 45}: 244023, (2012)
\bibitem{bal_invers}
P. Balazs, D. T. Stoeva, {\it Representation of the inverse of a frame multiplier}, J. Math. Anal. Appl., 422(2), 981-994, (2015)
\bibitem{batamitk}
F. Batayneh, M. Mitkovski, {\it Localized frames and compactness}, J. Fourier Anal. Appl. 22:568–590, (2016)
\bibitem{bellcorso}
G. Bellomonte, R. Corso,
{\it Frames and weak frames for unbounded operators}, arXiv:1812.10699v2, (2019)
\bibitem{gb_ct_riesz} G. Bellomonte and C. Trapani, {\em Riesz-like bases in Rigged Hilbert Spaces}, Zeitschr. Anal. Anwen., {\bf 35}, 243-265, (2016).
\bibitem{Ben_Pfander}
J. Benedetto, G. Pfander, {\it Frame expansions for Gabor multipliers}, Applied and Computational Harmonic Analysis, 20(1), 26–40, (2006).
\bibitem{christensen2} O. Christensen, {\em An Introduction to Frames and Riesz Bases}, Boston, Birkh\"auser, (2016)
\bibitem{Corso_seq}
R. Corso, {\it Sesquilinear forms associated to sequences on Hilbert spaces}, Monatshefte für Mathematik, 189(4), 625-650, (2019)
\bibitem{Corso_seq2}
R. Corso, {\it Generalized frame operator, lower  semi-frames and sequences of translates}, arXiv:1912.03261,   (2019)
\bibitem{daub1}
I. Daubechies, {\it Time-frequency localization operators: a geometric phase space approach}, IEEE
Trans. Inform. Theory, 34(4):605–612, (1988).
\bibitem{daub2}
I. Daubechies, T. Paul, {\it Time-frequency localization operators—a geometric phase space
 approach. II. The use of dilations}, Inverse Problems, 4(3):661–680, (1988).
\bibitem{FeiNow}
H. G. Feichtinger, K. Nowak, {\it A first survey of Gabor multipliers}, in: Advances in Gabor analysis, edited by H. G. Feichtinger and T. Strohmer, Boston Birkh\"{a}user, Applied and Numerical Harmonic Analysis, 99-128, (2003).
\bibitem{forgro}
M. Fornasier, K. Gr¨ochenig, {\it  Intrinsic localization of frames}, Constr. Approx. 22, 395–415 (2005)
\bibitem{forrau}
M. Fornasier, H. Rauhut, {\it Continuous frames, function spaces, and the discretization problem}, J. Fourier Anal. Appl. 11(3), 245–287 (2005)
\bibitem{gabhan}
J-P. Gabardo, D. Han, {\it Frames associated with measurable spaces}, Adv. Comput. Math., 18:127–147, 2003.
\bibitem{gavruta} L. \mbox{${\rm G\check{a}vru\c{t}a}$}, {\em Frames and operators}, Appl. Comp. Harmon. Anal. {\bf 32} (2012), 139-144.
\bibitem{gelf3}I.M. Gel'fand, G.E. Shilov, E. Saletan, {\em Generalized Functions\/}, Vol.III,
Academic Press, New York, 1967.
\bibitem{gelf} I. M. Gel'fand, N. Ya. Vilenkin,
{\em Generalized Functions}, Vol.IV, Academic Press, New York, 1964.
\bibitem{gro}
K. Gr\"{o}chenig, {\it Localization of frames, Banach frames, and the invertibility of the frame operator}, J. Fourier Anal. Appl. 10, 105–132 (2004)
\bibitem{gro2}
K. Gr\"{o}chenig, {\it Localized frames are finite unions of Riesz sequences}, Adv. Comput. Math. 18, 149–157 (2003)
\bibitem{Groechenig}
K. Gr\"{o}chenig, {\it Representation and approximation of pseudodifferential operators by sums of Gabor multipliers}, Appl. Anal., 90(3-4):385–401, (2010).
\bibitem{heil} C. Heil {\em A Basis Theory Primer}, Expanded Edition,   Birkh\"user/Springer, New York, (2011).
\bibitem{horvath}   J. Horvath,  {\em Topological Vector Spaces and Distributions}, Addison-Wesley, 1966.
\bibitem{kaiser} G. Kaiser, {\em A friendly guide to wavelets}, Birkh{\"a}user, Boston, (1994).
\bibitem{Kato}
T. Kato, {\it Perturbation Theory for Linear Operators}, Springer, Berlin, (1966).
\bibitem{lll}
F. Li, P. Li, A. Liu, {\it Decomposition of analysis operators and frame ranges for continuous frames}, Numerical Functional Analysis and Optimization, 37:2, 238-252, (2016)
\bibitem{roch}
R. Rochberg, {\it Toeplitz and Hankel operators, wavelets, NWO sequences, and almost diagonalization of
operators}, Operator Theory: Operator Algebras and Applications, Part 1 (Proc. Symp. Pure Mathematics
vol 51) (Providence, RI: American Mathematical Society) pp 425–44, (1990)
\bibitem{roch2}
R. Rochberg, {\it A correspondence principle for Toeplitz and Calder\'{o}n-Toeplitz operators}, Israel Math. Conf. Proc. 5 229–43 (1992)
\bibitem{Schm}
K. Schm\"{u}dgen, {\it Unbounded Self-adjoint Operators on Hilbert Space}, Springer, Dordrecht, (2012).
\bibitem{bal_ms_2}  M. Speckbacher, P. Balasz, {\em Frames, their relatives and reproducing
	kernel Hilbert spaces}, ArXiv: 1704.02818, 2017
\bibitem{bal_mult}
D. T. Stoeva, P. Balazs, {\it Invertibility of multipliers}, Appl. Comput. Harmon. Anal., 33(2), 292-299, (2012).
\bibitem{bal_uncond}
D. T. Stoeva, P. Balazs, {\it Canonical forms of unconditionally convergent multipliers}, J. Math. Anal. Appl., 399(1), 252–259, (2013).
\bibitem{bal_Riesz}	
D. T. Stoeva, P. Balazs, {\it Riesz bases multipliers}, In M. Cepedello Boiso, H. Hedenmalm,	M. A. Kaashoek, A. Montes-Rodr\'{i}guez, and S. Treil, editors, Concrete Operators, Spectral Theory, Operators in Harmonic Analysis and Approximation, vol 236 of Operator Theory: Advances and Applications, 475-482. Birkhäuser, Springer Basel, (2014).
\bibitem{bal_dual_frame2}
D. T. Stoeva, P. Balazs. {\it The dual frame induced by an invertible frame multiplier}. In Sampling Theory and Applications (SampTA), 2015 International Conference on, 101–104. IEEE, (2015).
\bibitem{bal_dual_frame}	
D. T. Stoeva, P. Balazs. {\it On the dual frame induced by an invertible frame multiplier}, Sampling Theory in Signal and Image Processing, 15, 119-130, (2016).
\bibitem{bal_survey}
D. T. Stoeva, P. Balazs,  {\it A survey on the unconditional convergence and the invertibility of multipliers with implementation}, In: Sampling - Theory and Applications (A Centennial Celebration of Claude Shannon), S. D. Casey, K. Okoudjou, M. Robinson, B. Sadler (Ed.), Applied and Numerical Harmonic Analysis Series, Springer,  (2020).
\bibitem{TTT} C. Trapani, S. Triolo, F. Tschinke, {\em Distribution Frames and Bases}, J. Fourier Anal. and Appl. 25, pp. 2109-2140 (2019)
\bibitem{trts} C. Trapani, F. Tschinke,
{\em Partial Multiplication of Operators in Rigged Hilbert},
Int. Equ. Operator Theory, 51, 4, 583–600, (2005)
\bibitem{tschinke}
F. Tschinke, {\em Riesz-Fischer maps, Semiframes and Frames in rigged Hilbert spaces},
 arXiv:1910.14447 (2019)
\end{thebibliography}

\end{document}